\documentclass[draft]{article} 

\usepackage{amsmath, amssymb, amsthm}
\usepackage{textcomp}
\usepackage{color}
\usepackage[dvipdfm,a4paper,left=30mm,right=30mm,top=35mm,bottom=35mm]{geometry}
\usepackage[titletoc, toc, title]{appendix}

\begin{document}

\theoremstyle{plain}
\renewenvironment{proof}{ \noindent {\bfseries Proof.}}{qed}
\newtheorem{defin}{Definition}
\newtheorem{theorem}{Theorem}[section]
\newtheorem{remark}{Remark}
\newtheorem{proposition}{Proposition}
\newtheorem{lemma}{Lemma}
\newtheorem{definition}{Definition}
\newtheorem{prop}{Proposition}[section]
\newtheorem{cor}{Corollary}
\def\begproof{\noindent{\bf Proof: }}
\def\endproof{\quad\vrule height4pt width4pt depth0pt \medskip}
\def\div{\nabla\cdot}
\def\rot{\nabla\times}
\def\sign{{\rm sign}}
\def\arsinh{{\rm arsinh}}
\def\arcosh{{\rm arcosh}}
\def\diag{{\rm diag}}
\def\const{{\rm const}}


\def\RR{\mathbb{R}}
\def\NN{\mathbb{N}}
\def\ZZ{\mathbb{Z}}
\def\DD{{\cal D}}
\def\FF{{\cal F}} 
\def\LL{{\cal L}} 
\def\SS{{\cal S}}

\def\F{\mathcal{F}}
\def\div{\mathrm {div}\,}
\def\supp{\mathrm {supp}\,}
\def\sgn{\mathrm {sgn}\,}
\def\H{\mathscr{H}^{n-1 }}
\def\L{\mathcal L}
\def\T{\, \mathcal{T}}

\def\pa{\partial}
\def\na{\nabla}
\def\eps{\varepsilon}
\def\vphi{\varphi}
\def\of{\overline{f}}
\def\bO{\overline{\Omega}}
\def\dO{\pa\Omega}
\def\And{\hspace*{.4cm}\text{and}\hspace*{.4cm}}
\def\Ox{\overline{x}}
\def\Oy{\overline{y}}
\def\LF{L^2( \Omega\times\RR^d, F^{-1}(v) \dxv)}
\def\LFt{L^{\infty}\big((0,T); L^2( \Omega\times\RR^d, F^{-1}(v) \dxv)\big)}
\def\Lf{L^2_{F^{-1}}}
\def\Delav{\big(-\Delta_v\big)^{\alpha/2}}
\def\Delax{\big(-\Delta_x\big)^{\alpha/2}}
\def\Dela{\big(-\Delta\big)^{\alpha/2}}
\def\iRn{\int_{\RR^d}}

\def\hx{\hat{x}}
\def\hv{\hat{v}}
\def\hw{\hat{w}}
\def\dtxv{\, \text{d}t \text{d}x \text{d}v }
\def\dtxw{\, \text{d}t \text{d}x \text{d}w }
\def\dxyv{\, \text{d}x \text{d}y \text{d}v }
\def\dxvw{\, \text{d}x \text{d}v \text{d}w }
\def\dtx{\, \text{d}t \text{d}x }
\def\dtv{\, \text{d}t \text{d}v }
\def\dxv{\, \text{d}x \text{d}v }
\def\dxy{\, \text{d}x \text{d}y }
\def\dyx{\, \text{d}y \text{d}x }
\def\dxz{\, \text{d}x ,\text{d}z }
\def\dvw{\, \text{d}v \text{d}w }
\def\dxw{\, \text{d}x \text{d}w }
\def\dv{\, \text{d}v}
\def\dw{\, \text{d}w}
\def\dx{\, \text{d}x}
\def\dy{\, \text{d}y}
\def\dz{\, \text{d}z}
\def\dt{\, \text{d}t}
\def\dxi{\, \text{d}\xi}
\def\dsig{\,\text{d}\sigma(x)}
\def\dhvhw{\, \text{d}\hat{v} \text{d}\hat{w} }
\def\dhv{\, \text{d}\hat{v}}
\def\rO{\mathring{\Omega}}
\def\nv{\langle n(x_b)|v\rangle}
\def\vn{\langle n|v\rangle}
\def\ts{\tilde{s}}
\def\la{\langle}
\def\ra{\rangle}
\def\ds{\displaystyle}
\def\tt{\tilde{\tau}}
\def\orho{\overline{\rho}}

\def\Fh{\widehat F}
\def\HSRs{\mathcal{H}_{SR}^s }
\def\Lsre{\mathscr{L}_{\text{\tiny{SR}},\eps}^*}
\def\Lsr*{\mathscr{L}_{\text{\tiny{SR}}}^*}

\def\eps{\varepsilon}
\def\phi{\varphi}
\newcommand{\Bchi}{\mbox{$\hspace{0.12em}\shortmid\hspace{-0.62em}\chi$}}
\def\C{\hbox{\rlap{\kern.24em\raise.1ex\hbox
      {\vrule height1.3ex width.9pt}}C}}
\def\R{\mathbb{R}}
\def\P{\hbox{\rlap{I}\kern.16em P}}
\def\Q{\hbox{\rlap{\kern.24em\raise.1ex\hbox
      {\vrule height1.3ex width.9pt}}Q}}
\def\M{\hbox{\rlap{I}\kern.16em\rlap{I}M}}
\def\N{\hbox{\rlap{I}\kern.16em\rlap{I}N}}
\def\Z{\hbox{\rlap{Z}\kern.20em Z}}
\def\K{\mathcal{K}}
\def\({\begin{eqnarray}}
\def\){\end{eqnarray}}
\def\[{\begin{eqnarray*}}
\def\]{\end{eqnarray*}}
\def\part#1#2{{\partial #1\over\partial #2}}
\def\partk#1#2#3{{\partial^#3 #1\over\partial #2^#3}}
\def\mat#1{{D #1\over Dt}}
\def\dt{\, \partial_t}
\def\as{_*}
\def\d{\, {\rm d}}
\def\e{ \text{e}}
\def\D{\mathcal{D}}
\def\feps{f_\eps}
\def\pmb#1{\setbox0=\hbox{$#1$}
  \kern-.025em\copy0\kern-\wd0
  \kern-.05em\copy0\kern-\wd0
  \kern-.025em\raise.0433em\box0 }
\def\bar{\overline}
\def\lbar{\underline}
\def\fref#1{(\ref{#1})}

\begin{center}
{\LARGE Fractional diffusion limit for a fractional\\[8pt]
 Vlasov-Fokker-Planck equation}
\bigskip

{\large P. Aceves-S\'anchez}\footnote{Fakult\"at f\"ur Mathematik, Universit\"at Wien.} and
{\large L. Cesbron}\footnote{DPMMS, Center for Mathematical Sciences, University of Cambridge.}
\end{center}
\vskip 1cm

\noindent{\bf Abstract.} This paper is devoted to the rigorous derivation of the macroscopic limit of a Vlasov-Fokker-Planck
equation in which the Laplacian is replaced by a fractional Laplacian. The evolution of the density is governed by a fractional
heat equation with the addition of a convective term coming from the external force. The analysis is performed by a modified test function
method and by obtaining a priori estimates from quadratic entropy bounds. In addition, we give the proof of existence and uniqueness of
solutions to the Vlasov-fractional-Fokker-Planck equation.

\vskip 1cm

\medskip

\vskip 1cm

\noindent{\bf Acknowledgment:} P.A.S. acknowledges support from Consejo Nacional
de  Ciencia y Tecnologia of Mexico, the PhD program {\em Dissipation and Dispersion
in Nonlinear PDEs} funded by the Austrian Science Fund, grant no. W1245, and
the Vienna Science and Technology Fund, grant no. LS13-029. L.C. acknowledges support from the ERC Grant Mathematical Topics of Kinetic Theory. 



\tableofcontents


\section{Introduction}

\subsection{The Vlasov-L\'evy-Fokker-Planck equation}

In this paper we investigate the long-time/small mean-free-path asymptotic behavior in the low-field case of the solution of the Vlasov-L\'evy-Fokker-Planck (VLFP) equation
\begin{subequations}
\begin{align}
 \partial_t f + v \cdot \nabla_x f + E \cdot \nabla_v f &= \nabla_v \cdot ( v f) - \Delav f & \text{ in } ( 0, \infty)\times \R^d \times \R^d, \label{eq:vlfpE} \\
                                            f( 0, x, v) &= f^{in} ( x, v )          & \text{ in } \R^d \times \R^d, \label{eq:vlfpEit}
\end{align}
\end{subequations}
where $\alpha \in [ 1, 2]$. This equation describes the evolution of the density of an ensemble of particles denoted as $f( t, x, v)$ in phase space, where
$t \geq 0$, $x \in \R^d$ and $v \in \R^d$ stand for, respectively, time, position and velocity. The operator $\Dela$ denotes the fractional
Laplacian and is defined by \eqref{def:fracLapInt}. Let us recall that, at a microscopic level, equation \eqref{eq:vlfpE}-\eqref{eq:vlfpEit} is related
to the Langevin equation

\begin{align}
 \d x ( t) &= v( t) \d t, \nonumber \\
 \d v ( t) &= - v( t) \d t + E \d t + \d L^\alpha_t, \label{eq:lang}
\end{align}
where $L^\alpha_t$ is a Markov process with generator $-\Dela$ and $( x( t), v( t))$ describe the position and velocity of a 
single particle (see \cite{Jourdain+2011} and \cite{Risken+1996}). Therefore, this models describes the position and velocity of a particle
that is affected by three mechanisms: a dragging force, an acceleration and a pure jump process. 


In the particular case when $\alpha = 2$ the fractional operator $\Dela$ takes the form of a Laplace operator $\Delta$ 
and \eqref{eq:vlfpE}-\eqref{eq:vlfpEit} reduces to the usual Vlasov-Fokker-Planck equation. In this case the Fokker-Planck operator is known to have an equilibrium
distribution function given by a Maxwellian $M( v) = C \exp{ \left( -| v|^2 \right)}$ where $C > 0$ is a normalization constant. The 
Vlasov-Fokker-Planck equation
has been used in the modeling of many physical phenomena, in particular, for the description of the evolution of plasmas \cite{Risken+1996}. 
However, there
are some settings in which particles may have long jumps and an $\alpha$-stable distribution process is more suitable to describe 
the phenomenon, see for instance \cite{Schertzer2001}.

The case in which $\alpha = 2$ reduces to the classical Vlasov-Fokker-Planck equation for a given external field. This equation is related to the
Vlasov-Poisson-Fokker-Planck system (VPFP) in the case in which the electric field is self-consistent. Questions such as existence 
of solutions, hydrodynamic limits and long time behaviour for the VPFP system has been extensively studied by many authors, see for instance \cite{Bouchut+1995}, \cite{Pfaffelmoser1992}, and \cite{Goudon+2005}. In particular, in \cite{ElGhani+2010} the low
field limit is studied for the VPFP system and a Drift-Diffusion-Poisson system is obtained in a rigorous manner.

Let us note that, although it is classical in the framework of kinetic theory to consider a self-consistence electric fields that expresses how particles repulse one another, one can also, in the VPFP system, consider the case in which particles are 
attracted by each other and this model is used in the description of galactic dynamics.

\medskip

In the rest of the paper we shall need the following notation: The fractional (or L\'evy) Fokker-Planck operator denoted by $\mathcal{L}^{\alpha/2}$ and defined as
\begin{equation} \label{def:LFP}
\mathcal{L}^{\alpha/2} f = \na_v \cdot \big( v f \big) - \Delav f. 
\end{equation}

In order to investigate the asymptotic behaviour of the system, we introduce the Knudsen number $\eps$ which represent the ratio between the mean-free-path and the observation length scale. In the case when $E = 0$ it was observed in \cite{Cesbron2012} that the time rescaling $t' \rightarrow \eps^{\alpha-1} t$ and 
introducing a factor $ 1 / \eps$ in front of $\L^{ \alpha / 2}$ is the appropriate scaling at which diffusion will be observed in the limit as $\eps$ goes to zero. Moreover, we
introduce 
the factor $1 / \eps^{ 2 - \alpha}$ in front of the force field term $E$ corresponding to a low-field limit scaling since we shall consider the case
$1\leq \alpha \leq 2$ and thus the scaling of the collision operator $1/\eps$ is much greater than the scaling of the electric field $1 / \eps^{ 2 - \alpha}$. Thus we shall study in this paper the asymptotic behaviour as $\eps$ tends to zero of the solutions of following rescaled VLFP equation

\begin{equation} \label{eq:vlfpeps}
\eps^{\alpha-1} \pa_t f^\eps + v\cdot \na_x f^\eps + \eps^{\alpha-2} E(t,x)\cdot\na_v f^\eps =  \frac{1}{\eps} \Big(\na_v\cdot \left( v f \right) - \Delav f \Big).
\end{equation}


\subsection{Preliminaries on the Fractional Fokker-Planck operator}

In this paper we denote by $\widehat f$ or $\F ( f)$ the Fourier transform of $f$ and define it as

\[
 \widehat f ( k) = \int_{ \R^d} \e^{ - i k \cdot x} f ( x) \d x. 
\]
There are several equivalent definitions of the fractional Laplacian in the whole domain (see \cite{Kwasnicki2015} or \cite{DiNezza+}). It can be defined via a 
Fourier multiplier as

\[
 \F \Big( \Dela ( f) \Big) ( k) = | k|^\alpha \F ( f) ( k). 
\]
On the other hand, assuming that $f$ is a rapidly decaying function we can define the 
fractional Laplacian in terms of a hypersingular integral as

\begin{equation}\label{def:fracLapInt}
\Delav ( f) ( v) = c_{ d, \alpha} \, \text{P.V.} \int_{ \R^d} \frac{ f( v) - f ( w) }{ | v - w|^{ d + \alpha}} \d w
\end{equation}
where P.V. denotes the Cauchy principal value and the constant $c_{ d, \alpha}$ is given by 

\begin{equation}\label{def:c}
 c_{ d, \alpha} = \frac{ 2^\alpha \Gamma \left( \frac{ d+\alpha}{ 2} \right)}{ 2 \pi^{ d/2} | \Gamma \left( - \frac{ \alpha}{ 2} \right)|},
\end{equation}
and $\Gamma ( \cdot)$ denotes the Gamma function. In \cite{DiNezza+} it is proven that for any $d > 1$, $c_{ d, \alpha} \to 0$ as $\alpha \to 2$.
Thus \eqref{def:fracLapInt} does not make sense if we take $\alpha = 2$. However, we have the following result.

\begin{proposition}
 Let $d > 1$. Then for any $f \in C^\infty_0 ( \R^d)$ we have
 
 \[
  \lim_{ \alpha \to 2} \Dela f = - \Delta f.
 \]
\end{proposition}
\noindent For an account of the properties of the fractional Laplacian consult \cite{DiNezza+}, \cite{Vazquez2014}, \cite{Stein1970} or \cite{Landkof1972}. Let us note
that due to its dependence on the whole domain, the fractional Laplacian is a nonlocal operator 
and it has the scaling property $\Delav ( f_\lambda) ( v) = \lambda^\alpha \Delav f ( \lambda v)$, for any $\lambda > 0$ where
$f_\lambda ( v) = f ( \lambda v)$. Since it will be useful later on in our analysis, we also mention that since the fractional Laplacian is an integro-differential operator it satisfies:
\[
\int \Dela f \dv = 0.
\]

In \cite{Biler+2003} it is proved that the L\'evy-Fokker-Planck operator $\L^{ \alpha / 2}$ defined by \eqref{def:LFP} has a unique normalized
equilibrium distribution that we shall denote by $G_\alpha$. Therefore, the Fourier transformation of $G_\alpha$ denoted as $\widehat{G_\alpha}$ and
defined as

\[
 \widehat{ G_\alpha} ( \xi) := \int_{ \R^d} \e^{ - i \xi \cdot v} G_\alpha ( v) \d v,   
\]
satisfies

\[
 \xi \cdot \nabla_\xi \widehat{ G_\alpha} + | \xi|^\alpha \widehat{ G_\alpha} = 0.
\]
Thus yielding

\begin{equation}\label{def:eqdisLFP}
 \widehat{ G_\alpha} ( \xi) = \e^{ - | \xi|^\alpha / \alpha}.
\end{equation}
In the jargon of stochastic analysis, random variables having a characteristic function of the form \eqref{def:eqdisLFP} are called symmetric $\alpha$-stable
random variables, consult \cite{Applebaum2009}. Using the notation of \cite{Bogdan+2007} let 
us note that setting $t = 1 / \alpha$, $x = v$, and $y=0$, we obtain the identity $G_\alpha ( v) = p ( 1/ \alpha, v, 0)$. Thus Lemma 3 of \cite{Bogdan+2007} states that there exists $C_1 = C_1 ( d, \alpha) > 0$ such that 

\begin{equation}\label{eq:Gbounds}
C_1^{ -1} \bigg( \frac{ 1}{ \alpha | v|^{ d + \alpha}} \wedge \frac{ 1}{ \alpha^{ d / \alpha}} \bigg) \leq G_\alpha ( v) \leq C_1 \bigg( \frac{ 1}{ \alpha | v|^{ d + \alpha}} \wedge \frac{ 1}{ \alpha^{ d / \alpha}} \bigg),
\end{equation}
for all $v \in \R^d$, where $a \wedge b$ denotes the minimum between $a$ and $b$. On the other hand, Lemma 5 of \cite{Bogdan+2007} states the existence
of a positive constant $C_2 = C_2 ( d, \alpha)$ such that

\begin{equation}\label{eq:gradGbounds}
\frac{ | v|}{ C_2} \bigg( \frac{ 1}{ \alpha | v|^{ d + 2 + \alpha}} \wedge \alpha^{ ( d + 2) / 2} \bigg) \leq \na_v \, G_\alpha ( v) \leq C_2 | v| \bigg( \frac{ 1}{ \alpha | v|^{ d + 2 + \alpha}} \wedge \alpha^{ ( d + 2) / 2} \bigg).
\end{equation}
 
\subsection{Main results}

As usually in the framework of fractional Vlasov-Fokker-Planck equations, we use the following definition of weak solutions:
\begin{definition} \label{def:weaksol}
Consider $f^{in}$ in $L^2(\R^d\times\R^d)$ and $E \in \big( W^{1,\infty}([0,T)\times\R^d) \big)^d$. We say that $f$ is a weak solution of \eqref{eq:vlfpE}-\eqref{eq:vlfpEit} if, for any $\phi\in\mathcal{C}^\infty_c([0,T)\times\R^d\times\R^d)$ 
\begin{equation} \label{eq:weakFracVFP}
\begin{aligned}
&\underset{Q_T}{\iiint} f \Big( \pa_t \phi + v\cdot \na_x \phi + \big(E(t,x) - v\big)\cdot\na_v \phi - \Dela \phi \Big) \dtxv \\
&\quad + \underset{\RR^d\times\RR^d}{\iint} f^{in}(x,v) \phi(0,x,v) \dxv = 0. 
\end{aligned}
\end{equation}
\end{definition}
Section 2 of this paper is devoted to a well-posedness result for the fractional Vlasov-Fokker-Planck with an external electric field $E$ in the following sense.
\begin{theorem} \label{thm:existence} 
For $f^{in}$ in $L^2(\R^d\times\R^d)$ and $E \in \big( W^{1,\infty}([0,T) \times \R^d) \big)^d$ there exists a unique weak solution $f$ of \eqref{eq:vlfpE}-\eqref{eq:vlfpEit} in the sense of Definition \ref{def:weaksol} and it satisfies
\begin{subequations}
\begin{align}
&f(t,x,v) \geq 0 \mbox{ on } Q_T , \label{eq:solposi} \\
&f \in \mathcal{X} := \bigg\{ f\in L^2(Q_T) : \frac{|f(t,x,v)-f(t,x,w)|}{|v-w|^{\frac{d+\alpha}{2}}} \in L^2(Q_T\times\RR^{d})\bigg\}. \label{eq:L2txHv}
\end{align}
\end{subequations}
\end{theorem}

\begin{remark}
The assumption $E \in \big( W^{1,\infty}([0,T) \times \R^d) \big)^d$ 
in Theorem \ref{thm:existence} is not optimal in the sense that we could replace it by $E \in \big( L^\infty ( [ 0, T) \times \R^d ) \big)^d$ or maybe it could be replaced by even weaker assumptions on $E$, however, finding the optimal regularity of $E$ is out of the scope of this paper.
\end{remark}

The proof of this existence result relies on using the Lax-Milgram theorem for a well chosen associated problem, in the spirit of the proof in \cite{Degond1986} and in \cite{Carrillo1998} for the existence of weak solutions of the Vlasov-Fokker-Planck equation. The proof of positivity \eqref{eq:solposi} is given in details as it involves the non-local nature of the fractional operator and, as such, differs from the classical proof. \\
In Section 3, we consider the electric field as a perturbation of the fractional Fokker-Planck operator and as such we introduce $\mathcal{T}_\eps$:
\begin{equation*}
\T_\eps ( f) := \na_v\cdot \Big[ \big( v- \eps^{\alpha-1} E(t,x) \big) f \Big] - \Delav f.
\end{equation*}
We prove existence and uniqueness of a normalized equilibrium $F_\eps$ for this perturbed operator in Proposition \ref{prop:equi}. Then, we follow the strategy introduced in \cite{AceSch}; we investigate the decay properties of this equilibrium and its convergence to the equilibrium of the unperturbed operator, $G_\alpha$, as $\eps$ goes to $0$ in Proposition \ref{prop:equi2}. Finally, we prove that $\mathcal{T}_\eps$ is dissipative with regards to the quadratic entropy, Proposition \ref{prop:posisemi}, which allows us to establish uniform boundedness results for $f_\eps$, the solution of the rescaled equation \eqref{eq:vlfpeps}-\eqref{eq:vlfpEit}, as well as its macroscopic density $\rho_\eps = \int f_\eps \dv$ and its distance to the kernel of $\mathcal{T}_\eps$ we which write $r_\eps$ defined by the expansion $f_\eps= \rho_\eps F_\eps + \eps^{\alpha/2} r_\eps$. 

In the last section, we turn to the proof of our main result which is the anomalous advection-diffusion limit of our kinetic model. We follow the method introduced in \cite{Cesbron2012} which consist in choosing a test function $\psi_\eps(t,x,v)$ which is solution, for some $\phi \in \mathcal{C}^\infty_c ([0,T)\times\R^d)$ of the auxiliary problem:
\begin{equation*}
\begin{array}{llr}
  &\eps v \cdot \nabla_x \psi_\eps - v \cdot \nabla_v \psi_\eps       =  0    \hspace{2cm} & \text{ in } [ 0, \infty) \times \R^d \times \R^d,\\
  &\psi ( t, x, 0)  =  \phi ( t, x)  & \text{ in } [ 0, \infty) \times \R^d,
\end{array}
\end{equation*}
and show that the weak formulation of our problem, \eqref{eq:weakFracVFP}, with such test functions converges to the weak formulation of the advection fractional diffusion equation. We first prove this convergence in the non-critical case, i.e. when $1<\alpha<2$ and then we turn to the critical cases $\alpha=1$ and $\alpha=2$. The outline of the proof remains the same in both critical cases but a few differences appear, for $\alpha=2$ the only difference is technical one in the study of the dissipative property of the perturbed operator whereas, in the case $\alpha=1$, we show that the equilibrium of the perturbed operator is independent of $\eps$ and as such it stays perturbed by the electric field $E(t,x)$ even in the macroscopic limit. In all cases, our main result reads:

\begin{theorem}\label{theo:main}
Let $\alpha$ be in $(1,2]$ and $f_\eps$ be the weak solution of \eqref{eq:vlfpeps}-\eqref{eq:vlfpEit} in the sense of Definition \ref{def:weaksol} on $[0,T)\times\R^d\times\R^d$ for some $T>0$ and  with $f^{ in} \in L^2_{G^{-1}_{\alpha}(v)}( \R^d \times \R^d )\cap L^1_+ ( \R^d \times \R^d)$. Then, $f_\eps$ converges weak-$\ast$ to $\rho(t,x) \, G_\alpha (v) $ in $L^\infty ( 0, T; L^2_{G_\alpha^{ -1} ( v)} ( \R^d \times \R^d ))$, where $\rho$ is the solution in the distributional sense of
\begin{equation} \label{macro}
\begin{array}{llr}
&  \dt \rho + \div ( E \rho) + ( - \Delta)^{\alpha/2} \rho  = 0   \hspace{2cm}       & \text{ in } [ 0, T) \times \R^d, \\
&  \rho( 0, x)  =  \rho^{in} ( x)  & \text{ in } \R^d,
\end{array}
\end{equation}
where $\rho^{ in} = \int f^{ in} \d v$. In the case $\alpha=1$ the same anomalous diffusion limit holds but instead of $G_\alpha(v)$ the equilibrium distribution of velocity becomes
\begin{equation}  \label{eq:GalphaE}
G_{\alpha,E} (t,x,v) = G_\alpha \big( v-E(t,x) \big)
\end{equation}
\end{theorem} 

The advection fractional-diffusion equation \eqref{macro} describes the evolution of the macroscopic density $\rho$ under the effect of a drift, consequence of the kinetic electric field, and a fractional diffusion phenomenon. The regularity of the solutions of this type of equations has been studied for instance in \cite{silvestre2011holder}, \cite{silvestre2012differentiability}, and \cite{Droniou+2006}. We refer the interested reader to those articles and references within for more details on this macroscopic model. 

\section{Existence of solution}

Throughout this paper, for any $T>0$ we write $Q_T = [0,T)\times\RR^d\times\RR^d$ and $\mathcal{C}_c^{\infty}(Q_T)$ the set of smooth function compactly supported in $Q_T$. This section is devoted to the proof of the following result of existence and regularity of weak solutions:
\begin{theorem} \label{thm:exist}
Consider $f^{in}$ in $L^2 (\RR^d\times\RR^d)$. There exists a unique weak solution $f$ of \eqref{eq:vlfpE} on $Q_T$ in the sense that for any $\phi\in\mathcal{C}_c^{\infty}(Q_T)$:
\begin{equation} \label{eq:weakFracVFP}
\begin{aligned}
&\underset{Q_T}{\iiint} f \Big( \pa_t \phi + v\cdot \na_x \phi + \big(E(t,x) - v\big)\cdot\na_v \phi - \Dela \phi \Big) \dtxv \\
&\quad + \underset{\RR^d\times\RR^d}{\iint} f^{in}(x,v) \phi(0,x,v) \dxv = 0 
\end{aligned}
\end{equation}
and this solution satisfies:
\begin{align}
&f(t,x,v) \geq 0 \mbox{ on } Q_T , \nonumber \\
&f \in \mathcal{X} := \bigg\{ f\in L^2(Q_T) : \frac{|f(t,x,v)-f(t,x,w)|}{|v-w|^{\frac{d+\alpha}{2}}} \in L^2(Q_T\times\RR^{d})\bigg\}. \label{eq:L2txHv}
\end{align}
\end{theorem}

\begin{remark}
Note that this definition of $\mathcal{X}$ is equivalent to saying that it is the set of functions which are in $L^2([0,T)\times\RR^d)$ with respect to time and position and in $H^{\alpha/2}(\RR^d)$ with respect to velocity.
\end{remark}

\begin{proof}
We follow the method in \cite{Degond1986} and in \cite{Carrillo1998} for the proof of existence and uniqueness of solutions to the linear Vlasov-Fokker-Planck equation. The first part of the proof consists in solving our linear problem in a variational setting, applying a well-known Lax-Milgram theorem of functional analysis. We consider the Hilbert space $\mathcal{X}$ provided with the norm 
\begin{equation} \label{def:normHs}
||f||_\mathcal{X} = \Bigg( ||f||^2_{L^2(Q_T)} + 2c_{d,\alpha}^{-1} ||(-\Delta)^{\frac{\alpha}{4}}f||^2_{L^2(Q_T)} \Bigg)^{\frac{1}{2}}
\end{equation}
where 
$c_{d,s}$ is defined in \eqref{def:c}. We refer the reader to \cite{DiNezza+} for properties of this functional space. Let us denote $\mathcal{T}$ the transport operator, given by 
\begin{equation*}
\mathcal{T}f = \pa_t f + v\cdot \na_x f - \big( v-E(t,x)\big) \cdot\na_v f.
\end{equation*}
We define the Hilbert space $\mathcal{Y}$ as:
\begin{equation} \label{def:hilY}
\mathcal{Y} = \bigg\{ f\in \mathcal{X} : \mathcal{T}f \in \mathcal{X}' \bigg\}
\end{equation}
where $\mathcal{X}'$ is the dual of $\mathcal{X}$. $(\cdot,\cdot)_{\mathcal{X},\mathcal{X}'}$ stands for the dual relation between $\mathcal{X}$ and its dual. $\mathcal{Y}$ is provided with the norm:
\begin{equation} \label{def:normY}
||f||^2_\mathcal{Y} = ||f||^2_{\mathcal{X}} + ||\mathcal{T}f||^2_{\mathcal{X}'}.
\end{equation}
In order to apply the Lax-Milgram theorem we consider the associated problem
\begin{equation} \label{eq:vlfpexp}
\begin{aligned}
&\pa_t \of + e^{-t} v\cdot \na_x \of + e^{t} E(t,x)\cdot\na_v \of + e^{\alpha t} \Dela \of + \lambda \of = 0 & (t,x,v) \in Q_T\\
& \of (0,x,v) = \of^{in} (x,v ) & (x,v) \in \RR^d\times\RR^d\\
\end{aligned}
\end{equation}
which comes formally by deriving \eqref{eq:vlfpE} for $\of = e^{-(\lambda+d)t} f\big(t,x,e^{-t}v\big)$ and $\of^{in} (x,v) = f^{in} (x,e^{-t} v)$ for some $\lambda\geq 0$. A weak solution of \eqref{eq:vlfpexp} is a function $\of \in \mathcal{X}$ such that for any $\phi$ in $\mathcal{C}_c^{\infty}(Q_T)$:
\begin{equation} \label{eq:wfvlfpexp}
\begin{aligned}
&\underset{Q_T}{\iiint} \Big( -\of \pa_t \phi - e^{-t} \of  v\cdot \na_x \phi - e^t \of E(t,x)\cdot \na_v \phi + e^{2st} \of \Dela \phi + \lambda \of \phi  \Big) \d t \d x \d v \\
&\hspace{1cm} - \underset{\RR^d\times\RR^d}{\iint} \of^{in} \phi(0,x,v) \d x \d v =0.
\end{aligned}
\end{equation}
We first prove existence of a solution in $\mathcal{X}$ of equation \eqref{eq:vlfpexp} and we will prove afterwards how this implies existence of a solution of the fractional Vlasov-Fokker-Planck equation with the electric field $E$. \\
We know that $\mathcal{C}_c^{\infty}(Q_T)$ is a subspace of $\mathcal{X}$ with a continuous injection (see, e.g. \cite{DiNezza+}) and we define the prehilbertian norm:
$$ |\phi|^2_{\mathcal{C}_c^{\infty}(Q_T)} = ||\phi||^2_{\mathcal{X}} +  \frac{1}{2} ||\phi(0,\cdot,\cdot)||^2_{L^2(\Omega\times\RR^d)}.$$
Now, we can introduce the bilinear form $a: \mathcal{X} \times \mathcal{C}_c^{\infty}(Q_T) \rightarrow \RR$ as:
$$ a(\of,\phi) = \underset{Q_T}{\iiint} \Big( -\of \pa_t \phi - e^{-t} \of  v\cdot \na_x \phi - e^t \of E(t,x)\cdot\na_v \phi+ e^{2st} \of \Dela \phi + \lambda \of \phi  \Big) \d t \d x \d v $$
and the continuous bounded linear operator $L$ on $\mathcal{C}_c^{\infty}(Q_T)$ given by:
$$ L(\phi) = -\underset{\RR^d\times\RR^d}{\iint} f^{in} ( x, v) \phi(0,x,v)\d x \d v .$$
To find a solution $\of$ in $\mathcal{X}$  of equation \eqref{eq:wfvlfpexp} is equivalent to finding a solution $\of$ in $\mathcal{X}$ of $a(\of,\phi) = L(\phi)$ for any $\phi \in \mathcal{C}_c^{\infty}(Q_T)$. Since $\of$ belongs to $\mathcal{X}$ it is easy to check that $a(\cdot,\phi)$ is continuous. To verify the coercivity of $a$ we write:
\begin{align*}
&-\underset{Q_T}{\iiint} \Big( \phi\pa_t \phi + e^{-t}\phi v\cdot \na_x \phi - e^t \phi E(t,x)\cdot\na_v \phi \Big) \d t \d x \d v  = \frac{1}{2} \underset{\RR^d\times\RR^d}{\iint} |\phi(0,x,v)|^2 \d x \d v
\end{align*}
and also:
\begin{align*}
\underset{Q_T}{\iiint} e^{2st} \phi \Dela \phi \d t \d x \d v  = \underset{Q_T}{\iiint} e^{2st} |(-\Delta)^\frac{s}{2}\phi |^2 \d t \d x \d v .
\end{align*}
Hence, we see that
\begin{align*}
a(\phi,\phi) &= \underset{Q_T}{\iiint} \bigg( \lambda \phi^2 + e^{2st} |(-\Delta)^\frac{ \alpha}{ 4}\phi |^2 \bigg) \d t \d x \d v  +  \frac{1}{2} \underset{\Omega\times\RR^d}{\iint} |\phi(0,x,v)|^2 \d t \d x \d v  
\end{align*}
which can be bounded from below as $a(\phi,\phi) \geq \min(1,\lambda) |\phi|_{\mathcal{C}_c^{\infty}(Q_T)}^2$. Thus, the Lax-Milgram theorem implies the existence of $\of$ in $\mathcal{X}$ satisfying \eqref{eq:wfvlfpexp}. Now, we want to show that this yields existence of a solution of \eqref{eq:weakFracVFP}. To that end, we first consider $\tilde{\phi}$ in $\mathcal{C}_c^{\infty}(Q_T)$ such that $\phi(t,x,v)= e^{\lambda t}\tilde{\phi}(t,x,e^{-t}v)$. Equation \eqref{eq:wfvlfpexp} becomes (writing $\tilde{\phi}(e^{-t} v)$ instead of $\tilde{\phi}(t,x,e^{-t}v)$)
\begin{align*}
&\underset{Q_T}{\iiint} e^{\lambda t} \Big( -\of \pa_t \tilde{\phi}(e^{-t}v) - \of e^{-t}  v\cdot \na_x \tilde{\phi}(e^{-t}v) + \of e^{-t} v \cdot\na_v \tilde{\phi}(e^{-t}v) - \of E(t,x)\cdot\na_v \tilde{\phi} (e^{-t} v) \\
&\hspace{1.5cm} +  \of \Dela \tilde{\phi}(e^{-t}v) \Big) \dtxv - \underset{\RR^d\times\RR^d}{\iint} f_{in} \tilde{\phi}(0,x,v) \dxv = 0.
\end{align*}
Hence, if we define $f(t,x,v) = e^{(\lambda+d)t} \of(t,x,e^t v)$ and change the variable $v\rightarrow e^{-t} v$, we recover equation \eqref{eq:weakFracVFP}. It is straightforward to check that $f$ is in $\mathcal{X}$ and it satisfies \eqref{eq:weakFracVFP} for any $\tilde{\phi}$ in $\mathcal{C}_c^{\infty}(Q_T)$. Moreover, since $ f\mapsto df  - \Dela f $ is a linear bounded operator from $\mathcal{X}$ to $\mathcal{X}'$, the transport term $\mathcal{T}f$ is in $\mathcal{X}'$, hence $f\in \mathcal{Y}$ and \eqref{eq:weakFracVFP} is verified in $\mathcal{X}'$. 

Since the VLFP equation is linear, to show uniqueness it is enough to show that the unique solution with zero initial data is the null function $f\equiv 0$. Let $f$ be a solution of this problem on $\mathcal{Y}$. As before, we define $\of = e^{-(\lambda + d)t} f(t,x,e^{-t}v)$, which satisfies equation \eqref{eq:vlfpexp} with $\of_{in}$ null. Since $f\in \mathcal{Y}$, we know that $\of$ belongs to $\mathcal{X}$ and,  moreover, that if we define $\widetilde{\mathcal{T}}$ as 
\begin{equation} \label{eq:defTtilde}
\widetilde{\mathcal{T}}\of= \pa_t \of + e^{-t} v\cdot \na_x \of +e^t E(t,x)\cdot\na_v \of
\end{equation}
then $\widetilde{\mathcal{T}}\of$ belongs to $\mathcal{X}'$. Through integration by parts we have
$$ 2 \big( \widetilde{\mathcal{T}} \of , \of \big)_{\mathcal{X}',\mathcal{X}} = \underset{\RR^d\times\RR^d}{\iint} \big( \of \big)^2(T,x,v) \d x \d v  \geq 0.$$
On the other hand, since $\of$ satisfies \eqref{eq:vlfpexp}, $\widetilde{\mathcal{T}} \of = - \lambda \of - \Dela \of$ in the sense of distributions which yields 
\begin{equation} \label{eq:ttildeof}
\big( \widetilde{\mathcal{T}} \of , \of \big)_{\mathcal{X}',\mathcal{X}} = - \underset{Q_T}{\iiint} \Big(\lambda \of^2 + e^{\alpha t} \big| (-\Delta)^{\frac{\alpha}{4}} \of \big|^2  \Big) \d t \d x \d v \leq 0.
\end{equation}
Hence both expression are null, in particular this means that the integral $\lambda \of^2$ is null, hence $f = \of \equiv 0$ a.e. on $Q_T$: the solution is unique. In order to prove the positivity of the solution consider once again the associated problem \eqref{eq:vlfpexp} and its solution $\of $ for some $\of^{in}\in L^2(\RR^d\times\RR^d)$ with $\of^{in} \geq 0$. Next, we define $\of_+$ and $\of_-$ the positive and negative parts of $\of$ given by:
\begin{align*}
&\of_+(t,x,v) = \max (f(t,x,v),0); & \of_-(t,x,v) = \max (-f(t,x,v), 0)
\end{align*}
so that $\of= \of_+ - \of_-$ and we denote by $A_+$ and $A_-$ the respective supports of $\of_+$ and $\of_-$. Using $\widetilde{\mathcal{T}}$ defined in \eqref{eq:defTtilde} we have through integration by parts
\begin{align*}
\big( \widetilde{\mathcal{T}}\of , \of_- \big) &= \underset{Q_T}{\iiint} \Big( \of_- \pa_t \big( \of_+ -\of_-\big)+ e^{-t} \of_- v\cdot \na_x \big( \of_+ -\of_-\big)\\
&\hspace{1,5cm}+ e^t \of_- E(t,x)\cdot\na_v\big( \of_+ -\of_-\big) \Big) \d t \d x \d v \\
&= - \frac{1}{2} \underset{\RR^d\times\RR^d}{\iint} \Big( \of_-^2(T,x,v) - \of_-^2(0,x,v) \Big) \d x \d v \\
&\hspace{1,5cm} + \underset{Q_T}{\iiint} \Big(  \of_- \pa_t \of_+ + e^{-t} \of_- v\cdot \na_x \of_+ + e^t \of_- E(t,x)\cdot\na_v \of_+ \Big) \d t \d x \d v.
\end{align*}
By definition of $\of_+$ and $\of_-$ we know that $A_+ \cap A_- = \emptyset$, hence wherever $\of_-$ is not zero, both $\pa_t \of_+$, $\na_x\of_+$ and $\na_v \of_+$ are naught, and vice-versa. Moreover, we assume $\of^{in} \geq 0$ which means $\of_-(0,x,v) = 0$ so that
\begin{equation*}
\big( \widetilde{\mathcal{T}}\of , \of_- \big) = - \frac{1}{2} \underset{\RR^d\times\RR^d}{\iint} \of_-^2(T,x,v) \d x \d v \leq 0.
\end{equation*}
Since $\of$ is solution of \eqref{eq:vlfpexp} we know that $\widetilde{\mathcal{T}} \of = - \lambda \of - \Dela \of$ in the sense of distributions which yields 
\begin{align*}
\big( \widetilde{\mathcal{T}}\of , \of_- \big) &= \underset{Q_T}{\iiint} \Big( -\lambda \of_- \big( \of_+ -\of_-\big) - \of_- \Dela \big( \of_+ -\of_-\big) \Big) \d t \d x \d v
\end{align*}
where
\begin{align*}
\underset{\RR^d}{\int} \of_- \Dela (\of_+) \dv 
&= \underset{\R^d}{\int} \of_- (v) \, c_{ d, \alpha} \mbox{ P.V. } \underset{\R^d}{\int}  \frac{\of_+(v)-\of_+(w)}{|v-w|^{d+\alpha}} \d w \d v \\
& = \underset{ A_-}{\int} \of_- (v) \, c_{ d, \alpha} \mbox{ P.V. } \underset{ A_+}{\int}  \frac{\of_+(v)-\of_+(w)}{|v-w|^{d+\alpha}} \d w \d v     \\
&= - c_{ d, \alpha} \underset{ A_-}{\int} \mbox{ P.V. } \underset{ A_+}{\int} \frac{\of_-(v) \of_+(w) }{|v-w|^{d+\alpha}} \d w \d v \leq 0.
\end{align*}
Note that this integral is well defined because $\of \in \mathcal{X}$. Hence, we have:
\begin{align*}
\big( \widetilde{\mathcal{T}}\of , \of_- \big) &=  \underset{Q_T}{\iiint }\Big( \lambda \of_-^2 - \of_- \Dela \of_+ + \big| (-\Delta)^{\alpha/4} \of_- \big|^2 \Big) \dtxv \geq 0 .
\end{align*}
This proves that $\big( \widetilde{\mathcal{T}}\of , \of_- \big) = 0$ which, in particular, means $\lambda \of_-^2 = 0$ and concludes the proof of positivity, and consequently the proof of Theorem \ref{thm:exist}.
\end{proof}

\section{A priori estimates}

Let us consider the operator $\T_\eps$ a perturbation of the fractional Fokker-Planck operator with an electric field $E(t,x)\in  \big( W^{1,\infty}( [0,T)\times\RR^d ) \big)^d$ defined as
\begin{equation} \label{def:Teps}
\T_\eps ( f_\eps) = \na_v\cdot \Big[ \big( v- \eps^{\alpha-1} E(t,x) \big) f_\eps \Big] - \Delav f_\eps.
\end{equation} 
We will prove the following:

\begin{prop} \label{prop:equi}
For any $\eps > 0$ fixed, there exists a unique positive equilibrium distribution $F_\eps$ solution of: 
\begin{equation}\label{eq:equi}
 \T_\eps (F_\eps) =  \na_v\cdot \Big[ \big( v- \eps^{\alpha-1} E(t,x) \big) F_\eps \Big] - \Delav F_\eps = 0, \hspace{1cm} \int_{ \R^d} F_\eps \d v = 1.
 \end{equation}
\end{prop}

\begin{proof}
The Fourier transform in velocity of the equilibrium equation \eqref{eq:equi} reads
\begin{equation*}
\xi \cdot \na_\xi \widehat{F_\eps} = - \Big( i \xi \cdot  \eps^{\alpha-1} E(t,x) + |\xi |^{\alpha} \Big) \widehat{F_\eps},
\end{equation*}
for which we can compute the explicit solution:
\begin{equation} \label{eq:Fouriersol}
\widehat{F_\eps} (t,x,\xi) = \kappa e^{ - i  \eps^{\alpha-1}  \xi \cdot E(t,x) - |\xi |^{\alpha}/ \alpha },
\end{equation}
where $\kappa$ is a positive constant which ensures the normalisation of the equilibrium. Now, although the inverse Fourier transform 
$\mathcal{F}^{-1} \big( \widehat{F_\eps} \big) (t,x,v)$ 
 is not explicit let us note that $F_\eps$ can be expressed as a translation of the equilibrium distribution $G_\alpha$ of the fractional Fokker-Planck operator:
\begin{equation}\label{eq:Feps}
 F_\eps (t,x,v) = G_\alpha \big( v - \eps^{\alpha-1} E(t,x) \big).
\end{equation}
Hence, the positivity and normalization of $F_\eps$ follows from the properties of $G_\alpha$.
\end{proof}

\begin{prop}\label{prop:equi2}
Let $F_\eps$ be the unique normalized equilibrium distribution of \eqref{def:Teps}. Then there exist positive constants 
$\mu$, $c_1$, $c_2$ and $c_3$ such that:
\begin{itemize}
 \item[\emph{(i)}] $\displaystyle c_1 G_\alpha \leq F_\eps \leq c_2 G_\alpha $,
 \item[\emph{(ii)}] $\displaystyle \bigg\lVert \frac{\pa_t F_\eps}{F_\eps} \bigg\lVert_{ L^\infty ( \d v \d x \d t)}, \bigg\lVert \frac{v\cdot \na_x F_\eps}{F_\eps} \bigg\lVert_{ L^\infty ( \d v \d x \d t)}  \leq \eps^{\alpha-1} \mu$, 
 \item[\emph{(iii)}] $\displaystyle | F_\eps - G_\alpha | \leq \eps^{ \alpha - 1} c_3 G_\alpha $.
\end{itemize}
for $\eps > 0$ small enough.
\end{prop}

\begin{proof} We shall start by proving part (i). Let us assume that $L$ is an arbitrary vector in $ \R^d$ 
such that $ | L| \leq 1$, then is easy to see that there exists $R_1 > 0$ big enough such that

\[
 \frac{ 1}{ 2^{ \frac{ 1}{ d + \alpha}}} \leq \bigg| 1 - \frac{ | L|}{ | v|} \bigg| \leq \bigg| \frac{ v}{ | v|} - \frac{ L}{ | v|} \bigg|,
\]
for all $| v| > R_1$. Hence, it follows that

\[
 \frac{ 1}{ | v - L|^{ d + \alpha}} \leq \frac{ 2}{ | v|^{ d + \alpha}},
\]
for all $| v| > R_1$. Thus, using \eqref{eq:Gbounds} we obtain that there exists $\widetilde{ C} > 0$ and $R > 0$ big enough such that

\[
 G_\alpha ( v - L) \leq \widetilde{ C} G_\alpha ( v),
\]
for all $ | v| > R$ and all $L \in \R^d$ with $L \leq 1$. Now, let $C_2 > 0$ such that 

\[
 C_2 \bigg( \min_{ v \in B ( 0, R)} G_\alpha ( v) \bigg) \geq \lVert G_\alpha \lVert_\infty,
\]
where $B ( 0, R) \subset \R^d$, is the ball of radius $R$ centered at the origin. Let us note that the minimum exists since $G_\alpha$ is continuous.
Thus choosing $\mu_2 = \widetilde{ C} \vee C_2$, where $a \vee b$ denotes the maximum between $a$ and $b$, we obtain

\[
 G_\alpha ( v - L) \leq \mu_2 G_\alpha ( v).
\]
Next, writing $w = v + L$ where $L \in \R^d$ with $| L| \leq 1$ we obtain

\[
 G_\alpha ( w) \leq \mu_1 G_\alpha ( w - L),
\]
Thus, taking $\mu_1 = 1 / \mu_2$ we obtain

\[
 \mu_1 G_\alpha ( v) \leq G_\alpha ( v - L),
\]
for all $v \in \R^d$ and $| L| \leq 1$. 

On the other hand, for part (ii), let us start by noting that thanks to \eqref{eq:Feps}, $F_\eps$ satisfies the following identities:

\[
\frac{\pa_t F_\eps}{F_\eps} = - \eps^{\alpha-1} \pa_t E(t,x) \cdot \frac{\na_v \, G_\alpha \big( v - \eps^{\alpha-1} E(t,x) \big)}{ G_\alpha \big( v - \eps^{\alpha-1} E(t,x) \big)},
\]
and
\[
\frac{v\cdot\na_x F_\eps}{F_\eps} = - \eps^{\alpha-1} \na_x E(t,x) \frac{ v \cdot \na_v \, G_\alpha \big( v - \eps^{\alpha-1} E(t,x) \big)}{G_\alpha \big( v - \eps^{\alpha-1} E(t,x) \big)}.
\]
Hence, thanks to the assumption $E \in W^{ 1, \infty} ([0,T)\times\RR^d)^d$ we only need to prove that there exists a $C > 0$ such that

\begin{equation}\label{eq:Gbound}
 | v \cdot \na_v \, G_\alpha ( v - L) | \leq C G_\alpha ( v - L),
\end{equation}
for all $v \in \R^d$, and all $L \in \R^d$ with $| L| \leq 1$. This follows via a similar line of reasoning as in the proof of part (i) around the control \eqref{eq:gradGbounds}.

Finally we prove part (iii). Since $G_\alpha$ is smooth by the mean value theorem we obtain

\begin{align}
 | F_\eps ( v) - G_\alpha ( v)| &= | G_\alpha ( v - \eps^{ \alpha - 1} E) - G_\alpha ( v) | \nonumber \\
                         &= \eps^{ \alpha - 1} | E | | \nabla_v \, G_\alpha ( v - \vartheta \, \eps^{ \alpha - 1} E )|, \nonumber
\end{align}
where $\vartheta \in ( 0, 1)$. Thus, the result follows thanks to \eqref{eq:Gbound} and since $E \in W^{ 1, \infty} ([0,T)\times\RR^d)^d$. 

\end{proof}

The key ingredient in order to obtain the a priori estimates needed to pass to the limit in \eqref{eq:vlfpeps} is the positivity of the dissipation which we state in the following result.

\begin{prop}\label{prop:posisemi}
 Let us consider the operator $\T_\eps$ defined by \eqref{def:Teps}. The associated dissipation, defined bellow, satisfies
 
 \begin{equation}\label{eq:dissipation}
 \mathcal{D}_\eps (f) :=  - \iint \T_\eps ( f) \frac{ f}{ F_\eps} \d v \d x = \iiint \bigg( \frac{f(v)}{F_\eps (v)} -\frac{f(w)}{F_\eps (w)} \bigg)^2 \frac{F_\eps(v)}{|v-w|^{d+\alpha}} \d w \d v \d x, 
 \end{equation}
and if we write $\rho(t,x) = \int f(t,x,v) \d v$, then for all $f \in L^2_{F_\eps^{ -1}} ( \R^d\times\R^d)$ we have
\begin{equation}\label{eq:disscont}
\mathcal{D}_\eps (f) \geq \int ( f - \rho F_\eps )^2 \frac{\d x \d v}{F_\eps (v)}.
\end{equation}
\end{prop}

\begin{proof}
The Poincar\'e type inequality \eqref{eq:disscont} is a particular case of the so-called $\Phi$-entropy inequalites introduced in \cite{Gentil+2008}. For the sake of completeness 
we shall give a sketch of the proof adapted to the case that we need.

We shall first start proving \eqref{eq:dissipation}. Writing $\Phi_\eps = v-\eps^{\alpha-1} E(t,x)$ and $g=f/F_\eps$, and since $F_\eps$ satisfies \eqref{eq:equi} we have:
\begin{align*}
 \mathcal{D}_\eps (f)  &= - \iint  \Big( \na_v\cdot \left( \Phi_\eps g F_\eps\right) g  - \Delav \left( gF_\eps \right)g\Big) \d v \d x \\
&=-  \iint \Big( \Phi_\eps F_\eps \frac{1}{2}\na_v (g^2) + \na_v\cdot (\Phi_\eps F_\eps) g^2 - \Delav(g) gF_\eps \Big) \d v \d x\\
&= \iint \Big( \frac{1}{2} g^2 \Delav (F_\eps) - g^2 \Delav (F_\eps) + g\Delav (g) F_\eps \Big)\d v \d x\\
&= \iint \Big(  g\Delav (g) -\frac{1}{2} \Delav(g^2) \Big) F_\eps \d v \d x .
\end{align*}
Hence, using \eqref{def:fracLapInt} we see that:
\begin{align*}
&\iint \Big( g\Delav (g) - \frac{1}{2} \Delav (g^2)  \Big) F_\eps \d v \d x \\
& \hspace{1cm}= \iiint \bigg(\frac{g(v) \big( g(v) - g(w) \big)}{|v-w|^{d+\alpha}}   - \frac{1}{2} \frac{g^2(v) - g^2(w)}{|v-w|^{d+\alpha}}  \bigg) F_\eps(t,x,v) \d w \d v \d x\\
& \hspace{1cm}=\frac{1}{2} \iiint \frac{\big( g(v)-g(w)\big)^2}{|v-w|^{d+\alpha}} F_\eps(t,x,v) \d w \d v \d x .
\end{align*}
Recall that $F_\eps(t,x,v) = G_\alpha \big( v - \eps^{\alpha-1} E(t,x) \big)$, therefore through a simple change of variable, if we call $h(t,x,v) = g\big( v - \eps^{\alpha -1} E(t,x) \big)$ we have:
\begin{align*}
 \mathcal{D}_\eps (f)  &= \frac{1}{2} \iiint \frac{\big( h(t,x,v)-h(t,x,w)\big)^2}{|v-w|^{d+\alpha}} G_\alpha(v) \d w \d v \d x .
\end{align*}
In order to prove the control \eqref{eq:disscont} we consider the semigroup associated with $\Dela$ 
\begin{equation}
\frac{\d }{\d t} P_t (h)(v) = - \Dela  \Big( P_t (h) \Big)(v)
\end{equation}
with $P_0 (h)(v) = h(v)$ and we see, using  \eqref{eq:Fouriersol}, that if we introduce the kernel 
$$ K_t (v)=  \mathcal{F}^{-1} \Big( \kappa e^{ - t |\xi |^{\alpha}/ \alpha }\Big) (v)$$
where $\kappa$ is a constant normalizing $K_1$, then we have explicitly $P_t(h) = K_t \ast h$. For $s\in [0,t]$ we consider
\begin{equation} 
\psi(s) = P_s ( H^2)(v)
\end{equation}
with $H = P_{t-s} (h)$. We then have for $s\in[0,t]$:
\begin{align*}
\psi'(s) &= \frac{\d }{\d s} \bigg[ K_s \ast \Big( K_{t-s} \ast h \Big)^2 \bigg] \\
&=\Big(\frac{\d}{\d s} K_s \Big) \ast \Big( K_{t-s} \ast h \Big)^2 + K_s \ast \frac{\d}{\d s} \Big[ \big( K_{t-s} \ast h \big)^2 \Big]\\
&= P_s \Big(- \Dela H^2 \Big) + 2 P_s \Big( H \Dela H \Big) \\
&= P_s \bigg( \int \frac{\big(H(v)-H(w) \big)^2 }{|v-w|^{d+\alpha}} \dw \bigg)
\end{align*}
Using the integral expression of the convolution and Jensen's inequality it is straightforward to see that $\big( P_{t-s}(h)(v) - P_{t-s}(h)(w) \big)^2 \leq P_{t-s} \big(h(v)-h(w)\big)^2$. Therefore, using Fubini's theorem, we have:
\begin{align*}
\psi'(s)(v) &\leq P_s \bigg( P_{t-s} \bigg( \int \frac{\big(h(v)-h(w)\big)^2}{|v-w|^{d+\alpha}} \d w \bigg) \bigg) = P_t  \bigg( \int \frac{\big(h(v)-h(w)\big)^2}{|v-w|^{d+\alpha}} \d w \bigg).
\end{align*}
Integrating over $s\in[0,t]$ one gets
\begin{align*}
P_t \big( h^2\big)(v) - \Big( P_t (h)(v) \Big)^2 \leq t P_t  \bigg( \int \frac{\big(h(v)-h(w)\big)^2}{|v-w|^{d+\alpha}} \d w \bigg).
\end{align*}
Finally, taking $t=1$ and evaluating at $v=0$ we get:
\begin{equation}
\int h^2(w)  G_\alpha (w) \d w - \bigg( \int h(w) G_\alpha (w) \d w\bigg)^2 \leq \iint \frac{\big(h(v)-h(w)\big)^2}{|v-w|^{d+\alpha}}  G_\alpha  (v) \d v\d w.
\end{equation}
Through a simple change of variables, inverse of the one we did earlier, we obtain
\begin{equation} \label{eq:fracSob}
\int g^2(w)  F_\eps(w) \d w - \bigg( \int g(w) F_\eps (w) \d w\bigg)^2 \leq \iint \frac{\big(g(v)-g(w)\big)^2}{|v-w|^{d+\alpha}}  F_\eps  (v) \d v\d w.
\end{equation}
Finally, replacing $g$ by $f/F_\eps$, since $F_\eps$ is normalized, we recover \eqref{eq:disscont}.

\end{proof}

Since the operator $\T_\eps$ is negative semidefinite in $L^2_{F_{\eps}^{-1}} ( \R^d )$ it is natural to look for bounds of the quadratic entropy associated to
solutions $f_\eps$ of \eqref{eq:vlfpeps}. We gather the appropriate a priori estimates that we shall need to pass to the limit in \eqref{eq:vlfpeps} in
the following result.

\begin{prop}\label{prop:apriori}
 Let the assumptions of Theorem \ref{theo:main} be satisfied and let $f_\eps$ be the solution of \eqref{eq:vlfpeps}. We introduce the residue $r_\eps$ through the macro-micro decomposition $f_\eps = \rho_\eps F_\eps + \eps^{ \alpha / 2} r_\eps$. Then, uniformly in $\eps \in (0,1)$, we have:
 \begin{itemize}
  \item[\emph{(i)}] $( f_\eps)$ is bounded in $L^\infty ([0,T) ; L^2_{G_{\alpha}^{ -1} ( v)} ( \R^d \times \R^d))$ and
  in $L^\infty ( [0,T); L^1 ( \R^d\times\R^d))$,
  \item[\emph{(ii)}] $( \rho_\eps)$ is bounded in $L^\infty ( [0,T); L^2 ( \R^d))$,
  \item[\emph{(iii)}] $( r_\eps)$ is bounded in $L^2 ([0,T) ; L^2_{G_{\alpha}^{-1}(v)} (\R^d \times\R^d ))$.
 \end{itemize}
\end{prop}

\begin{proof}
 Multiplying \eqref{eq:vlfpeps} by $f_\eps / F_\eps$, integrations by parts yield

 \[
  \frac{ \eps^{ \alpha-1}}{ 2} \frac{ \d }{ \d t} \underset{\R^d\times\R^d}{\iint}\frac{ f^2_\eps}{ F_\eps} \d v \d x  + \frac{ \eps^{\alpha-1}}{ 2} \underset{\R^d\times\R^d}{\iint}\frac{ f_\eps^2}{ F_\eps} \frac{ \dt F_\eps}{ F_\eps} \d v \d x - \frac{ 1}{ 2} \underset{\R^d\times\R^d}{\iint} \frac{ f^2_\eps}{ F_\eps} \frac{ v \cdot \nabla_x F_\eps}{ F^2_\eps} \d v \d x + \frac{1}{\eps }\mathcal{D}_\eps (f^\eps) =0.
 \]
 Thus, thanks to Proposition \ref{prop:equi2}, part (i) and (ii), and \eqref{eq:disscont} we obtain
 
 \begin{equation}\label{eq:auxi}
  \frac{ \eps^\alpha}{ 2} \frac{ \d }{ \d t} \underset{\R^d\times\R^d}{\iint} \frac{ f^2_\eps}{ F_\eps} \d v \d x + \underset{\R^d\times\R^d}{\iint} \frac{ (f_\eps - \rho_\eps F_\eps)^2}{ F_\eps} \d v \d x \leq \eps^\alpha \mu\underset{\R^d\times\R^d}{\iint}\frac{ f^2_\eps}{ F_\eps} \d v \d x.
 \end{equation}
 Whence, part (i) follows by Gronwall's lemma and the fact that the weights $1/G_{\alpha}$ and
 $1/F_\eps$ are equivalent uniformly in $\eps$ which follows from Proposition \ref{prop:equi2}, part (i). On the other hand, part (ii) follows thanks to
 the inequality
 
 \[
  \rho_\eps \leq \bigg( \int \frac{ f^2_\eps}{ F_\eps} \d v \bigg)^{ 1/2},
 \]
 which is an immediate consequence of Cauchy-Schwarz and the fact $\int F_\eps \d v = 1$. Finally, part (iii) follows from \eqref{eq:auxi} after integrating with respect to $t$ over $(0,T)$ and thanks to Proposition \ref{prop:equi2} part (ii).

\end{proof}

\section{Proof of Theorem \ref{theo:main}}

We shall follow the method introduced in \cite{Cesbron2012}. Let us start by introducing the following auxiliary problem: for $\phi \in \mathcal{C}_c^{\infty} ([0,T)\times\R^d)$, define $\psi_\eps$ the unique solution of

\begin{equation} \label{eq:auxeq}
\begin{array}{llr}
  &\eps v \cdot \nabla_x \psi_\eps - v \cdot \nabla_v \psi_\eps       =  0    \hspace{2cm} & \text{ in } [ 0, \infty) \times \R^d \times \R^d,\\
  &\psi ( t, x, 0)  =  \phi ( t, x)  & \text{ in } [ 0, \infty) \times \R^d
\end{array}
\end{equation}
The function $\psi_\eps$ can be obtained readily via the method of 
characteristics and can be expressed in an explicit manner as follows:

\begin{equation}\label{aux:function}
 \psi_\eps ( t, x, v) = \phi ( t, x + \eps v).
\end{equation}
Next, multiplying \eqref{eq:vlfpeps} by $\psi_\eps$ and through integrations by parts we obtain

\begin{align}
 & \underset{Q_T}{\iiint} f_\eps \Big( \eps^{ \alpha-1} \dt \psi_\eps + v \cdot \nabla_x \psi_\eps -  \frac{1}{\eps}( v - \eps^{ \alpha - 1} E ) \cdot \nabla_v \psi_\eps - \frac{1}{\eps}( - \Delta)^{ \alpha / 2} \psi_\eps \Big) \d v \d x \d t \nonumber \\
 & \hspace{5cm} + \eps^{\alpha-1} \underset{\R^d\times\R^d}{\iint} f^{ in} ( x, v) \psi_\eps ( 0, x, v) \d v \d x =0 \, . \label{eq:weakform}
\end{align}
Let us note the following
\begin{align}
 ( - \Delta_v)^{ \alpha / 2} \psi_\eps ( t, x, v) = \eps^{ \alpha} ( - \Delta _v)^{ \alpha / 2} \phi ( t, x + \eps v),\label{eq:auxiden} \\
 \na_v \psi_\eps (t,x,v) = \eps \na \phi(t, x + \eps v),
\end{align}
 which follows after a simple computation using the definition \eqref{def:fracLapInt} of the fractional Laplacian. Thus using the auxiliary equation \eqref{eq:auxeq} and plugging \eqref{eq:auxiden} into \eqref{eq:weakform} yields
\begin{align}
& \int_0^\infty \iint f_\eps \Big( \dt \phi ( t, x + \eps v) + E \cdot \nabla_x \phi ( t, x + \eps v) - ( - \Delta_v)^{ \alpha / 2} \phi ( t, x + \eps v) \Big) \d v \d x \d t \nonumber \\
& \hspace{7cm} + \iint f^{ in} ( x, v) \phi ( 0, x + \eps v) \d v \d x =0 \, . \label{eq:weakform1}
\end{align}

\subsection{The non-critical case: $1< \alpha < 2$}

In order to pass to the limit in this weak formulation, we introduce the following two results.

\begin{lemma}\label{lem:cv2}
 Let $( f_\eps)$ be the sequence of solutions of \eqref{eq:vlfpeps}, and $\rho$ be the limit of $(\rho_\eps)$ which exists thanks to Proposition \ref{prop:apriori} part (ii), then
 
 \[
  f_\eps(t,x,v) \rightharpoonup \rho(t,x) G_\alpha (v) \quad \text{ weakly in } L^\infty ([0,T); L^2_{G^{-1}_{\alpha}(v)} ( \R^d \times \R^d))
  \]
\end{lemma}

\begin{proof}
This lemma follows directly from Proposition \ref{prop:apriori}. Since $f_\eps$ is uniformly bounded, it converges weakly in $L^\infty([0,T);L^2_{G^{-1}_{\alpha}(v)}(\R^d\times\R^d))$. From the bounds on $F_\eps$ established in Proposition \ref{prop:equi2} and the boundedness of $\rho_\eps$ in $L^\infty( [0,T); L^2(\R^d))$ we see that $\rho_\eps(t,x) F_\eps (v)$ converges to $\rho(t,x) G_\alpha(v)$ weakly in $L^\infty([0,T);L^2_{G^{-1}_{\alpha}(v)}(\R^d\times\R^d))$ where $\rho$ is the weak limit of $\rho_\eps$. Finally, since the residue $r_\eps$ is bounded, it follows from the micro-macro decomposition $f_\eps = \rho_\eps F_\eps + \eps^{\alpha/2} r_\eps$ that the limit of $f_\eps$ is the same as the limit of $\rho_\eps F_\eps$. 
\end{proof}

\begin{lemma} \label{lem:cv}
For all test functions $\psi$ in $C^\infty_c ( [ 0, \infty) \times \R^d)$ we have:
\begin{equation} \label{lemeq:cvpsi}
\underset{\eps \rightarrow 0}{\lim} \underset{Q_T}{\iiint} f^\eps(t,x,v) \psi(t,x+\eps v) \d t \dx \dv = \underset{[0,T)\times\R^d}{\iint} \rho(t,x) \psi(t,x) \d x \d t.
\end{equation}
Moreover, if $E(t,x) \in W^{1,\infty}([0,T)\times\R^d)^d$ then for all $\Psi \in C^\infty_c ( [ 0, \infty) \times \R^d ; \R^d)$ the following convergence holds:
\begin{equation} \label{lemeq:cvEpsi}
\underset{\eps \rightarrow 0}{\lim} \underset{Q_T}{\iiint} f^\eps(t,x,v) E(t,x)\cdot \Psi(t,x+\eps v) \d t \d x \d v = \underset{[0,T)\times\R^d}{\iint} \rho(t,x) E(t,x) \cdot \Psi(t,x) \d x \d t.
\end{equation}
\end{lemma}

\begin{proof}
We will give a detailed proof of the convergence in \eqref{lemeq:cvEpsi}, the convergence in \eqref{lemeq:cvpsi} follows as a consequence of \eqref{lemeq:cvEpsi} 
by taking $\psi(t,x+\eps v) = E(t,x)\cdot \Psi(t,x+\eps v)$ with a smooth $E$ and Lemma \ref{lem:cv2}. For \eqref{lemeq:cvEpsi}, we write:

\begin{align}
 \underset{Q_T}{\iiint} f_\eps E(t,x) \cdot \Psi( t, x + \eps v) \d v \d x \d t  &=  \underset{[0,T)\times\R^d}{\iint} \rho(t,x) E(t,x) \cdot \Psi ( t, x)  \d x \d t \nonumber \\
 & \quad + \underset{Q_T}{\iiint} \Big( f_\eps - \rho(t,x) G_{\alpha}(v)\Big) E(t,x) \cdot \Psi(t,x) \d v \d x \d t \nonumber \\
 & \quad + \underset{Q_T}{\iiint} f_\eps E(t,x)\cdot \Big(\Psi ( t, x + \eps v) - \Psi ( t, x) \Big) \d v \d x \d t. \label{eq:auxpassage}
\end{align}
The second term in the right hand side of \eqref{eq:auxpassage} converges to zero since $f_\eps$ converges to $\rho G_\alpha$ weakly in $L^\infty ([0,T); L^2_{G^{-1}_{\alpha}(v)} ( \R^d \times \R^d))$ thanks to Lemma \ref{lem:cv2}. For the 
third term on the right hand side of \eqref{eq:auxpassage} thanks to Cauchy-Schwarz and H\"older we obtain
\begin{align}
 & \bigg| \underset{Q_T}{\iiint} f_\eps E(t,x) \cdot (\Psi ( t, x + \eps v) - \Psi ( t, x) ) \d v \d x \d t \bigg| \nonumber \\
 &  \leq \int_0^T \bigg(\underset{\R^d\times\R^d}{\iint}\frac{ f_\eps^2}{ G_{\alpha}} \d v \d x \bigg)^{ 1/2} \bigg( \underset{\R^d\times\R^d}{\iint} \Big[ E(t,x) \cdot ( \Psi ( t, x + \eps v) - \Psi ( t, x))\Big]^2 G_{\alpha} \d v \d x \bigg)^{ 1/2} \d t \nonumber \\
 &  \leq \lVert f_\eps \lVert_{ L^\infty ( [0,T); L^2_{G_{\alpha}^{ -1} ( v)} ( \R^d \times \R^d))} \nonumber \\
 & \hspace{1cm} \times \int_0^T \bigg(\underset{\R^d\times\R^d}{\iint} [ E(t,x) \cdot ( \Psi ( t, x + \eps v) - \Psi ( t, x))]^2 G_{\alpha} \d v \d x \bigg)^{ 1/2} \d t. \label{eq:auxpassage1} 
\end{align}
Next, let $R$ be an arbitrary positive real number and let us consider the following splitting
\begin{align}
 &\underset{\R^d\times\R^d}{\iint} \Big[ E \cdot ( \Psi ( t, x + \eps v) - \Psi ( t, x))\Big]^2 G_{\alpha} ( v) \d v \d x \nonumber \\
 & \hspace{4cm} =  \underset{\R^d}{\int} \, \underset{|v| \leq R}{\int} \Big[ E \cdot ( \Psi ( t, x + \eps v) - \Psi ( t, x))\Big]^2 G_{\alpha} ( v) \d v \d x \nonumber \\
 & \hspace{4cm} \quad +  \underset{\R^d}{\int} \, \underset{|v| > R}{\int} \Big[ E \cdot ( \Psi ( t, x + \eps v) - \Psi ( t, x))\Big]^2 G_{\alpha} ( v) \d v \d x.
\end{align}
We will use the regularity of $\Psi$ to bound the integral on $| v| < R$. To that end, let us consider the $\eps R$ neighborhood of the support of $\Psi$ denoted as $\Omega ( \eps R)$ which consists of the union of all the balls
of radius $\eps R$ having as center a point in $\supp \Psi$. Next, let $\varLambda$ denote the diameter of $\supp \Psi$ defined as the maximum over all the distances
between two points in $\supp \Psi$. Then it is clear that $\Omega ( \eps R) \subseteq B ( x_0; \varLambda + \eps R)$ where $B ( x_0; \varLambda + \eps R)$ denotes the ball with
center at $x_0$ and radius $\varLambda + \eps R$ and $x_0$ is any arbitrary fix point in 
$\supp \Psi$. Then for the integral over $| v| < R$ we have the following
\begin{align}
 & \underset{\R^d}{\int} \, \underset{|v| \leq R}{\int} [ E \cdot ( \Psi ( t, x + \eps v) - \Psi ( t, x))]^2 G_{\alpha}( v) \d v \d x \nonumber \\
 & \hspace{2cm} \leq \lVert G_{\alpha}\lVert_{L^\infty(\R^d)}   \underset{\R^d}{\int} \, \underset{|v| \leq R}{\int}\bigg( \sum_{ j=1}^d | E_j |\big|\eps v \cdot  \nabla_x  \Psi_j (t, x + \theta_j \eps v)\big| \bigg)^2 \d v \d x  \nonumber \\ 
  & \hspace{2cm} \leq 2 \eps^2  \lVert G_{\alpha}\lVert_{L^\infty(\R^d)}   \underset{\R^d}{\int} \, \underset{|v| \leq R}{\int}| v|^2 \bigg( \sum_{ j=1}^d | E_j |^2 \big| \nabla_x \Psi_j (t, x + \theta_j \eps v)\big|^2 \bigg) \d v \d x \nonumber \\
  & \hspace{2cm} \leq 2 \eps^2  \lVert G_{\alpha}\lVert_{L^\infty(\R^d)}  \lVert E \lVert^2_{ W^{ 1, \infty} ([0,T)\times\R^d)} \lVert\na_x \Psi \lVert_{L^\infty(\R^d)} \int_{ | v| \leq R} \int_{ B ( x_0, \delta + \eps R)} | v|^2 \d x \d v \nonumber \\
  & \hspace{2cm} \leq \eps^2 C_2 ( \varLambda + \eps R)^d R^{ d + 2}, \label{eq:vsmall}
\end{align}
where $ C_2$ is a constant depending on $\lVert E \lVert^2_{W^{ 1, \infty} ([0,T)\times\R^d)}$, $\lVert G_{\alpha} \lVert_{L^\infty(\R^d)}$ and $\lVert D^2_x \phi \lVert_{L^\infty(\R^d)}$ but not on $\eps$, and
$\theta_j \in ( 0, 1)$ for $j = 1, \ldots, d$ is such that $\Psi_j (t,x+\eps v) - \Psi_j (t,x) = \eps v \cdot \na_x \Psi_j (t,x+ \theta_j \eps v)$.  
For the integral on $| v| > R$ we use the decay of the equilibrium $G_{\alpha}(v)$ to derive the following upper bound:
\begin{align}
 &  \underset{\R^d}{\int} \, \underset{|v| > R}{\int} \Big[ E \cdot ( \Psi ( t, x + \eps v) - \Psi ( t, x))\Big]^2 G_{\alpha} ( v) \d v \d x \nonumber \\
 & \hspace{2.5cm} \leq \lVert E \lVert^2_{ W^{ 1, \infty} ([0,T)\times\R^d)} \underset{ | v| > R}{\int} \Bigg( \underset{\R^d}{\int} \Big( 2 | \Psi (t, x + \eps v) |^2 + 2 | \Psi (t, x) |^2 \Big) \d x \Bigg) G_{\alpha} ( v) \d v \nonumber \\
 & \hspace{2.5cm} \leq 4 \lVert E \lVert^2_{ W^{ 1, \infty} ([0,T)\times\R^d)} \underset{ \R^d}{\int} | \Psi ( t,x)|^2 \d x \underset{ | v| > R}{\int} G_{\alpha}( v) \d v \nonumber \\
 & \hspace{2.5cm} \leq C \underset{ | v| > R}{\int} G_{\alpha}( v) \d v. \nonumber
\end{align}
Thanks to Proposition \ref{prop:equi}, for any $\eta > 0$ we can choose $R > 0$ big enough
such that 
\[
 \bigg| G_{\alpha}( v) - \frac{ \vartheta}{ | v|^{ d + \alpha}} \bigg| \leq \frac{ \eta}{ | v|^{ d + \alpha} }, \qquad \text{ for all } | v| \geq R.
\]
Thus choosing $\eta = \vartheta$ we have the following estimate:
\begin{align}
\underset{ | v| > R}{\int} G_{\alpha} ( v) \d v & \leq \underset{ | v| > R}{\int} \bigg| G_{\alpha} ( v) - \frac{ \vartheta}{ | v|^{ d + \alpha} } \bigg| \d v + \underset{ | v| > R}{\int} \frac{ \vartheta}{ | v|^{ d + \alpha}} \d v  \nonumber \\
                              & \leq 2 \underset{ | v| > R}{\int} \frac{ \vartheta}{ | v|^{ d + \alpha}} \d v \nonumber \\
                              & \leq \frac{ C}{ R^{ \alpha}}. \nonumber
\end{align}
From which we conclude 
\begin{equation}\label{eq:vbig}
 \underset{\R^d}{\int} \, \underset{|v| > R}{\int} \Big[ E \cdot ( \Psi ( t, x + \eps v) - \Psi ( t, x))\Big]^2 G_{\alpha} ( v) \d v \d x \leq \frac{ C_2}{ R^\alpha}. 
\end{equation}
Next let us note that for any $\delta > 0$ we can choose $\widetilde{R} > 0$ such that $C_2 / R^\alpha < \delta / 2$ for all $R > \widetilde{R}$ and
then choose $\eps > 0$ so that $\eps^2 C_1 ( \varLambda + \eps R)^d R^{ d + 2} < \delta / 2$. And thus deduce that for $\eps$ small enough we have
\[
 \eps^2 C_1 ( \varLambda + \eps R)^d R^{ d + 2} + \frac{ C_2}{ R^\alpha} < \delta.
\]
Therefore, plugging \eqref{eq:vsmall} and 
\eqref{eq:vbig} into \eqref{eq:auxpassage1} and using Proposition \ref{prop:apriori}, part (i), we obtain that
there exists a fixed $C > 0$ such that
\begin{align}
& \bigg| \underset{Q_T}{\iiint} f_\eps E \cdot ( \Psi ( t, x + \eps v) - \Psi ( t, x)) \d v \d x \d t \bigg| \nonumber \\
& \hspace{3cm} \leq C \bigg( \eps^2 C_1 ( \varLambda + \eps R)^d R^{ d + 2} + \frac{ C_2}{ R^\alpha} \bigg) \nonumber \\
& \hspace{3cm} \leq C \delta, \nonumber 
\end{align}
for any $\delta > 0$, hence concluding that the third term on the right hand side of \eqref{eq:auxpassage} goes to zero as $\eps \rightarrow 0$. 
\end{proof}

Using Lemma \ref{lem:cv} we can now take the limit in \eqref{eq:weakform1} and conclude that $\rho$ satisfies

\[
\underset{[0,T)\times\R^d}{\iint} \rho \Big( \pa_t \phi + E \cdot \na_x \phi  - \Delax \phi \Big) \d x \d t + \underset{\R^d}{\int} \rho_{in}(x) \phi(0,x) \d x =0,
\]
for all $\phi \in C^{\infty}_c ( [ 0, T) \times \R^d )$. Thus concluding the proof of Theorem \ref{theo:main}.

%
%

\subsection{The critical cases $\alpha=1$ and $\alpha=2$}

\underline{In the critical case $\alpha=2$} we recover the classical Fokker-Planck operator which means, in particular, as mentioned in the Introduction, that its equilibrium is a Maxwellian $M( v) = C \exp{ \left( -| v|^2 \right)}$ instead of the heavy-tail distribution $G_\alpha$. We can still consider the perturbed operator $\mathcal{T}_\eps$ of Proposition \ref{prop:equi} and its equilibrium will also be a translation of the unperturbed one:
\begin{align*}
F_\eps (t,x,v) = C e^{- | v-\eps E(t,x) |^2 }
\end{align*} 
and since the decay of the Maxwellian is much faster than the decay of the heavy-tail distributions, Proposition \ref{prop:equi2} holds. The dissipative properties of the Fokker-Planck operator are well known, see e.g. \cite{CesHut} \cite{Goudon+2005} or \cite{BostanGoudon2008}, and it is straightforward to check the boundedness results of Proposition \ref{prop:apriori}. Hence, Lemma \ref{lem:cv2} holds and we can take the limit in the weak formulation \eqref{eq:weakform1} to prove that Theorem \ref{theo:main} holds in the case $\alpha = 2$. 

\medskip

\underline{In the critical case $\alpha=1$}, the perturbed operator $\T_\eps$ of \eqref{def:Teps} and its equilibrium $F_\eps$ \eqref{eq:Feps} lose their dependence with respect to $\eps$:
\begin{align*}
 &\T_\eps (f_\eps) = \T_E (f_\eps) =  \na_v\cdot \Big[ \big( v- E(t,x) \big) f_\eps \Big] - \Delav f_\eps, \\
 &F_\eps(t,x,v) = G_{1,E} (t,x,v) = G_{1} \big(v - E(t,x) \big).
\end{align*}
In particular, the equilibrium $G_{1,E}$ will remain unchanged in the limit as $\eps$ goes to $0$ and Proposition \ref{prop:equi2} will hold with $\alpha=1$ which, in particular, means that the bounds in $(ii)$ and $(iii)$ do not go to zero. The operator is still dissipative since the dependence on $\eps$ does not matter in the proof of Proposition \ref{prop:posisemi}, hence we still have \eqref{eq:fracSob} and multiplying \eqref{eq:vlfpeps} by $f_\eps /G_{1,E}$ and integrating by parts yields:
 \begin{equation}\label{eq:auxi}
  \frac{ \eps }{ 2} \frac{ \d }{ \d t} \underset{\R^d\times\R^d}{\iint} \frac{ f^2_\eps}{ G_{1,E}} \d v \d x + \underset{\R^d\times\R^d}{\iint} \frac{ (f_\eps - \rho_\eps G_{1,E})^2}{ G_{1,E}} \d v \d x \leq \eps \mu\underset{\R^d\times\R^d}{\iint}\frac{ f^2_\eps}{ G_{1,E}} \d v \d x.
 \end{equation}
Since $E$ is in $\big( W^{1,\infty}([0,T)\times\R^d)\big)^d$, if $f_\eps(t,\cdot,\cdot)$ is in $L^2_{G_{1,E}(t,x,v)}(\R^d\times\R^d)$ and bounded independently of time, then it is also in $L^2_{G_1(v)}(\R^d\times\R^d)$. As a consequence, from \eqref{eq:auxi} we still have the uniform in $\eps$ boundedness of $f_\eps$, $\rho_\eps = \int f_\eps \dv$ and the residue $r_\eps$ in $L^\infty([0,T); L^2_{G_1(v)} (\R^d \times\R^d))$ as stated in Proposition \ref{prop:apriori}. This yields the following modified version of Lemma 1:
\begin{lemma}\label{lem:cv2bis}
Let $\alpha=1$, $ (f_\eps)$ be the sequence of solutions of \eqref{eq:vlfpeps}, and $\rho$ be the limit of $(\rho_\eps)$ which exists thanks to Proposition \ref{prop:apriori} part (ii), then
 
 \[
  f_\eps(t,x,v) \rightharpoonup^\star \rho(t,x) G_{1,E} (t,x,v) \quad \text{ in } L^\infty ([0,T); L^2_{G_{1}^{-1}(v)} ( \R^d \times \R^d)).
  \]
\end{lemma}
Finally, for the proof of convergence of the weak formulation \eqref{eq:weakform1}, i.e. the proof of Lemma \ref{lem:cv}, we proceed essentially the same way. The only slight difference is that in order to control the third term of \eqref{eq:auxpassage} we will use Cauchy-Schwarz as in \eqref{eq:auxpassage1} but we multiplying and divide by $G_1(v)^{1/2}$ instead of the natural equilibrium $G_{1,E}$. The rest of the proof remains the same and we can then take the limit in the weak formulation, which concludes the proof of Theorem \ref{theo:main} with $\alpha = 1$.

\bibliographystyle{siam}

\bibliography{bibAceCes.bib}

\end{document}